\documentclass[9pt]{article}
\usepackage{amsmath, amssymb, amsfonts, amsthm,amscd}
\usepackage{amssymb, graphicx}
\usepackage{graphicx}
\usepackage{young}
\usepackage{youngtab}
\usepackage{psfrag}
\usepackage[all, 2cell]{xy}
\usepackage{bm}
\usepackage{amssymb}
\usepackage{epsfig}
\usepackage{epstopdf}
\usepackage{mathrsfs}
\usepackage{amscd}
\usepackage{amsopn}
\usepackage{latexsym,amsmath, amsfonts, graphics, epsf, epic}
\usepackage{amssymb}
\usepackage[latin1]{inputenc}
\usepackage[english]{babel}
\usepackage{amsthm}
\usepackage{graphicx}
\usepackage{amsfonts}
\usepackage{amsmath}
\usepackage{latexsym}
\usepackage{amscd}
\usepackage{verbatim}

\allowdisplaybreaks[1]

\newtheorem{teorem}{Theorem}[section]
\newtheorem{proposition}[teorem]{Proposition}
\newtheorem{lemma}[teorem]{Lemma}

\theoremstyle{remark}

\usepackage{color}
\definecolor{orange}{rgb}{1,.549,0}
 \definecolor{GreenYellow }{rgb}{ 0.15,   0.69, 0} % PANTONE 388
\definecolor{Yellowone}{rgb}{ 0, 1., 0} % PANTONE YELLOW
\definecolor{Goldenrod }{rgb}{  0, 0.10, 0.84} % PANTONE 109
\definecolor{Dandelion }{rgb}{ 0, 0.29, 0.84} % PANTONE 123
\definecolor{Apricot }{rgb}{ 0, 0.32, 0.52} % PANTONE 1565
\definecolor{Peach }{rgb}{ 0, 0.50, 0.70} % PANTONE 164
\definecolor{GreenYellow}{cmyk}{0.15,0,0.69,0}
\definecolor{RoyalPurple}{cmyk}{0.75,0.90,0,0}
\definecolor{Yellow}{cmyk}{0,0,1,0}
\definecolor{BlueViolet}{cmyk}{0.86,0.91,0,0.04}
\definecolor{Goldenrod}{cmyk}{0,0.10,0.84,0}
\definecolor{Periwinkle}{cmyk}{0.57,0.55,0,0}
\definecolor{Dandelion}{cmyk}{0,0.29,0.84,0}
\definecolor{CadetBlue}{cmyk}{0.62,0.57,0.23,0}
\definecolor{Apricot}{cmyk}{0,0.32,0.52,0}
\definecolor{CornflowerBlue}{cmyk}{0.65,0.13,0,0}
\definecolor{Peach}{cmyk}{0,0.50,0.70,0}
\definecolor{MidnightBlue}{cmyk}{0.98,0.13,0,0.43}
\definecolor{Melon}{cmyk}{0,0.46,0.5,0}
\definecolor{NavyBlue}{cmyk}{0.94,0.54,0,0}
\definecolor{YellowOrange}{cmyk}{0,0.42,1,0}
\definecolor{RoyalBlue}{cmyk}{1,0.50,0,0}
\definecolor{Orange}{cmyk}{0,0.61,0.87,0}
\definecolor{Blue}{cmyk}{1,1,0,0}
\definecolor{BurntOrange}{cmyk}{0,0.51,1,0}
\definecolor{Cerulean}{cmyk}{0.94,0.11,0,0}
\definecolor{Bittersweet}{cmyk}{0,0.75,1,0.24}
\definecolor{Cyan}{cmyk}{1,0,0,0}
\definecolor{RedOrange}{cmyk}{0,0.77,0.87,0}
\definecolor{ProcessBlue}{cmyk}{0.96,0,0,0}
\definecolor{Mahogany}{cmyk}{0,0.85,0.87,0.35}
\definecolor{SkyBlue}{cmyk}{0.62,0,0.12,0}
\definecolor{Maroon}{cmyk}{0,0.87,0.68,0.32}
\definecolor{Turquoise}{cmyk}{0.85,0,0.20,0}
\definecolor{BrickRed}{cmyk}{0,0.89,0.94,0.28}
\definecolor{TealBlue}{cmyk}{0.86,0,0.34,0.02}
\definecolor{Red}{cmyk}{0,1,1,0}
\definecolor{Aquamarine}{cmyk}{0.82,0,0.30,0}
\definecolor{OrangeRed}{cmyk}{0,1,0.50,0}
\definecolor{BlueGreen}{cmyk}{0.85,0,0.33,0}
\definecolor{RubineRed}{cmyk}{0,1,0.13,0}
\definecolor{Emerald}{cmyk}{1,0,0.50,0}
\definecolor{WildStrawberry}{cmyk}{0,0.96,0.39,0}
\definecolor{JungleGreen}{cmyk}{0.99,0,0.52,0}
\definecolor{Salmon}{cmyk}{0,0.53,0.38,0}
\definecolor{SeaGreen}{cmyk}{0.69,0,0.50,0}
\definecolor{CarnationPink}{cmyk}{0,0.63,0,0}
\definecolor{Green}{cmyk}{1,0,1,0}
\definecolor{Magenta}{cmyk}{0,1,0,0}
\definecolor{ForestGreen}{cmyk}{0.91,0,0.88,0.12}
\definecolor{VioletRed}{cmyk}{0,0.81,0,0}
\definecolor{PineGreen}{cmyk}{0.92,0,0.59,0.25}
\definecolor{Rhodamine}{cmyk}{0,0.82,0,0}
\definecolor{LimeGreen}{cmyk}{0.50,0,1,0}
\definecolor{Mulberry}{cmyk}{0.34,0.90,0,0.02}
\definecolor{YellowGreen}{cmyk}{0.44,0,0.74,0}
\definecolor{RedViolet}{cmyk}{0.07,0.90,0,0.34}
\definecolor{SpringGreen}{cmyk}{0.26,0,0.76,0}
\definecolor{Fuchsia}{cmyk}{0.47,0.91,0,0.08}
\definecolor{OliveGreen}{cmyk}{0.64,0,0.95,0.40}
\definecolor{Lavender}{cmyk}{0,0.48,0,0}
\definecolor{RawSienna}{cmyk}{0,0.72,1,0.45}
\definecolor{Thistle}{cmyk}{0.12,0.59,0,0}
\definecolor{Sepia}{cmyk}{0,0.83,1,0.70}
\definecolor{Orchid}{cmyk}{0.32,0.64,0,0}
\definecolor{Brown}{cmyk}{0,0.81,1,0.60}
\definecolor{DarkOrchid}{cmyk}{0.40,0.80,0.20,0}
\definecolor{Tan}{cmyk}{0.14,0.42,0.56,0}
\definecolor{Purple}{cmyk}{0.45,0.86,0,0}
\definecolor{Gray}{cmyk}{0,0,0,0.50}
\definecolor{Plum}{cmyk}{0.50,1,0,0}
\definecolor{Black}{cmyk}{0,0,0,1}
\definecolor{Violet}{cmyk}{0.79,0.88,0,0}
\definecolor{White}{cmyk}{0,0,0,0}

 \definecolor{rltred}{rgb}{0.75,0,0}
   \definecolor{rltgreen}{rgb}{0,0.5,0}
   \definecolor{oneblue}{rgb}{0,0,0.75}
   \definecolor{marron}{rgb}{0.64,0.16,0.16}
   \definecolor{forestgreen}{rgb}{0.13,0.54,0.13}

   \definecolor{purple}{rgb}{0.62,0.12,0.94}
   \definecolor{dockerblue}{rgb}{0.11,0.56,0.98}
   \definecolor{freeblue}{rgb}{0.25,0.41,0.88}
   \definecolor{myblue}{rgb}{0,0.2,0.4}
 \definecolor{Melon}{rgb}{ 0.46, 0.50, 0}

 \definecolor{Melone}{rgb}{ 0, 0.46, 0.50}

\begin{document}
\title{On covariants in exterior algebras for the even special orthogonal group}
\author{Salvatore Dolce\footnote{Universit\`a La Sapienza di Roma}}
\maketitle
Let $G:=SO(2n)$ be the even special orthogonal group over $\mathbb{C}$ and let $M_{2n}^+$ (resp. $M_{2n}^-$) be the space of symmetric (resp. skew-symmetric) complex matrices with respect to the usual transposition. 

We study the structure of the space $B^+:=\left(\bigwedge (M_{2n}^{+})^*\otimes M_{2n}^-\right)^G$, the space of $G-$equivariant skew-symmetric matrix valued alternating multilinear maps on the space of symmetric $n-$tuples of matrices, with $G$ acting by conjugation.

We prove that $B^+$ is a free module over a certain subalgebra of invariants $A:=\left(\bigwedge (M_{2n}^{+})^*\right)^G$ of rank $2n$. We give an explicit description for the basis of this module. Furthermore we prove new trace polynomial identities for symmetric matrices.  

Finally we show, using a computation made with the LiE software, that the analogous module $B^-:=\left(\bigwedge (M_{2n}^{+})^*\otimes M_{2n}^+\right)^G$ doesn't satisfy a similar property.

\section*{Introduction}

The main goal of this paper is to study the structure of the space $B:=\left(\bigwedge (M_{2n}^+)^*\otimes M_{2n}\right)^G$, which could be seen as the space of the $G-$equivariant matrix valued alternating multilinear maps on the space of symmetric $n-$tuples of matrices, with $G=SO(2n)$ the even special orthogonal group acting by conjugation.

We put on $B$ a structure of algebra and then we see $B$ as a module over the algebra of invariants $A:=\left(\bigwedge (M_{2n}^+)\right)^G$ in a natural way. For this purpose, let us recall some classical notations and conventions (for more details see \cite{Dolce}).

By the {\em antisymmetrizer} we mean the operator that sends a multilinear application\\ $f(x_1,\ldots,x_h)$ into the antisymmetric application $ \sum_{\sigma\in  S_h} \epsilon_\sigma f(x_{\sigma(1)},\ldots,x_{\sigma(h)})$. The main example for us is obtained applying the antisymmetrizer to the noncommutative monomial $x_1\cdots x_h$. We get the standard polynomial of degree $h$, which is by definition:  $$St_h(x_1,\ldots,x_h):=\sum_{\sigma\in  S_{h}}\epsilon_\sigma   x_{\sigma(1)} \dots x_{\sigma(h)}.$$
Let $R$ be any algebra (not necessarily associative) over a field $\mathbb{F}$, and let $V$ be a finite dimensional vector space  over $\mathbb F$. The set of  multilinear  antisymmetric maps from $V^k$ to $R$ can be identified in a natural way with $\bigwedge^kV^*\otimes R$.  Using the algebra structure of $R$ we get a product of these maps (which we denote by $\wedge$); for $G\in \bigwedge^hV^*\otimes R,\ H\in \bigwedge^kV^*\otimes R$ we define
$$(G\wedge H)( v_1,\ldots,v_{h+k}):=\frac{1}{h!k!}\sum_{\sigma\in S_{h+k}}\epsilon_\sigma G (v_{\sigma(1)},\ldots,v_{\sigma(h)})H(v_{\sigma(h+1) },\ldots,v_{\sigma(h+k)}) $$
$$=\sum_{\sigma\in  S_{h+k}/ S_{h}\times  S_{k}}\epsilon_\sigma G (v_{\sigma(1)},\ldots,v_{\sigma(h)})H(v_{\sigma(h+1) },\ldots,v_{\sigma(h+k)}).$$
It is easy to show (see \cite{Pr1}) that $St_a\wedge St_b=St_{a+b} $.

With this multiplication 
the algebra of multilinear  antisymmetric maps from $V $ to $R$ is isomorphic to the tensor product algebra $\bigwedge V^*\otimes R$.  

Assume now that $R$ is an associative algebra and $V\subset R$.  The inclusion map $X:V\to R$ is of course antisymmetric, since the symmetric group on one variable is trivial, hence $X\in \bigwedge V^*\otimes R$. By iterating the definition of wedge product we have that as a multilinear function, each power  $X^a:=X^{\wedge a}$  equals the standard polynomial $St_a$  computed in $V$.

We apply previous considerations to $\mathbb{F}=\mathbb{C}$, $R=M_{2n}$ and $V=M_{2n}^\pm$ (symmetric or skew-symmetric matrices);  the group $G=SO(2n)$ acts on this space, and hence on functions, by conjugation and it is interesting to study the algebra of $G$-equivariant maps 
$$B:=\left(\bigwedge \left(M_{2n}^\pm\right)^*\otimes M_{2n}\right)^G.$$
Furthermore we can clearly decompose $B$ as the direct sum $B\simeq B^+\oplus B^-$, where $B^\mp:= \left(\bigwedge (M_{2n}^{\bullet})^*\otimes M_{2n}^{\pm}\right)^G.$

The main part of this paper is devoted to study the spaces of type $B^+$ for symmetric matrices (the skew-symmetric one was already studied in \cite{PDP}), in particular their natural structure as a module over the algebra of invariants $A:=\left(\bigwedge (M_{2n}^{+})^*\right)^G$.

By a classical result on the algebra structure of invariants $A$, which we recover in the paper, we have that $A$ is the symmetric product of the exterior algebra in certain elements $T_i$ (where $i=0,...,n-1$) of degree $4i+1$ and the one dimensional algebra generated by an element $Q$ of even degree $2n$, with $Q\wedge Q=0$ (see \cite{Tak}).

We prove that $B^+$ is a free module of rank $2n$ over the subalgebra $A_{n}\subset A$ generated by $T_0,...,T_{n-2},Q$. Further we write an explicit basis for this module and the explicit relations for the product with the missing invariant. Let us remark that also in the case of skew-symmetric matrices in \cite{PDP} was proved that $B^+$ satisfies an analogous statement.

Finally, using a computer computation made with the software LiE, we show that the case $B^-$ cannot have an analogous property.

Let us remark that spaces analogous to $B$ for the other classical groups have been already studied. In fact in \cite{Dolce} the spaces $B:=\left(\bigwedge (M_{m}^{\pm})^*\otimes M_{m}\right)^G$, with $G$ the symplectic or the odd special orthogonal, appear as particular cases of covariants of infinitesimal symmetric spaces such that cohomology of the associated compact symmetric space is an exterior algebra.

\section{The main case: $\left(\bigwedge (M_{2n}^+)^*\otimes M_{2n}^-\right)^G$}
\subsection{Dimension}
Let us recall that there is a natural $GL(V)-$equivariant isomorphism $M_{2n}^+\simeq S^2(V)$, where $V=\mathbb{C}^{2n}$ and $G$ acts on $V$ by $g\cdot v:=gv$ (the usual product matrix by vector). We always identify these representations in the following. As in \cite{Dolce} we index the irreducible representations of $GL(V)$ by Young diagrams with at most $2n$ columns (the row of length $k$ corresponds to $\bigwedge ^kV$). The irreducible corresponding to the diagram $\lambda$ will be denoted by $S_{\lambda}(V)$. We decompose (see \cite{MD}) the space $\bigwedge (M_{2n}^+)^*$ as direct sum of irreducible representation of type  $H_{\underline a}^+(V)$ (with respect to the linear group) and study the invariants with respect to the even special orthogonal group. Let us recall that the Young diagram  $\lambda(\underline a)$ corresponding to the representation $H_{\underline a}^+(V)$ is built by nesting   the  {\em hook} diagrams $h_{a_i}$  whose column is of length $a_i+2$ and whose row is of the length $a_i+1$.

As an example, the diagram corresponding to $\lambda(4,3,1)$ is
$$
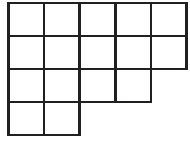
\begin{Young}
      &&&&\cr
     &&&& \cr
     &&& \cr
     &\cr
\end{Young}$$

In this case to compute the invariants, in addiction to compute the diagrams $H^+$ with even columns, we have to consider the diagrams which have the first row of maximal length and, when we delete it, which of these have even columns. In fact we can state:
\begin{lemma}
We have that $dim (S_{\tilde{\lambda}}(V))^{SO(2n)}=1$ if and only if $\tilde{\lambda}=2\lambda,2\lambda+1^m$, where $m=dim(V)$.
\end{lemma}
\begin{proof}
Let us consider $m$ copies of $V$ which we display as $V\otimes W$, dim$(W)=m$. In this setting the linear group $GL(m)$ acts on this vector space by tensor action on the second factor, commuting with the $SO(V)$ action on $V$. By Cauchy's formula (see \cite{Pr2}),
$$S(V\otimes W)^{SO(V)}=\bigoplus_{\lambda,\, ht(\lambda)\leq m} S_{\lambda}(V)^{SO(V)}\otimes S_{\lambda}(W).$$
On the other hand we know (see \cite{Pr2}) that the special orthogonal invariants of $m$ copies of $V$ are generated by the scalar products $(u_i,u_j)$ or the determinants $u_{i_1}\wedge\cdots\wedge u_{i_m}$. As a representation of $GL(m)=GL(W)$ the ring they generate is $$S\left(S^2(W)\right)\oplus \left(S\left(S^2(W)\right)\otimes\bigwedge^m (W)\right).$$ We recall (see \cite{Pr2}) that as a representation of $GL(m)$:
$$S(S^2(W))\otimes\bigwedge^m (W)\simeq \bigoplus_{\lambda}S_{2\lambda +1^m}(W),$$
so we have
$$S(V\otimes W)^{SO(V)}=\bigoplus_\lambda S_{2\lambda}(W)\oplus S_{2\lambda +1^m}(W).$$
 then our claim follows.
\end{proof}

So we can prove that
\begin{lemma}\label{evendim}
The dimension of the space of invariants $(\bigwedge M_{2n}^+)^G$ is $2^{n+1}$.
\begin{proof}
We know that diagrams of the first type (even rows) correspond to the sequences\\
 $2n>a_1>a_1-1>a_2>a_2-1>...\ge 0$, with $a_i$ even. So their number is $2^n$. 
 
On the other hand diagrams of the second type correspond to the sequences  
$$2n-1>a_1>a_1-1>...\ge 0,$$ with $a_i$ even. So as before their number is $2^n$. By adding their contributes we have the thesis.

\end{proof}
\end{lemma}
\begin{proposition}\label{dimpari}
The dimension of the space $(\bigwedge (M_{2n}^+)^*\otimes M_{2n}^-)^G$ is $(2n)2^{n}$.
\begin{proof}
Let us identify
$$\bigwedge (M_{2n}^+)^*\otimes M_{2n}^-\simeq \bigwedge (S^2V)\otimes \bigwedge^2 V\simeq \bigoplus_{\underline{a}}H_{\underline{a}}^+(V)\otimes \bigwedge^2(V)$$
then by using the Pieri's formula we know (see \cite{Pr2}) how to decompose each $H_{\underline{a}}^+(V)\otimes \bigwedge^2(V)$:
$$H_{\underline{a}}^+(V)\otimes \bigwedge^2(V)=\bigoplus_{\lambda\in\{\underline{a}\}_2}S_{\lambda}(V),$$
where $\{\underline{a}\}_2:=\{\lambda|\ \underline{a}\subset \lambda,\ |\lambda|=|\underline{a}|+2\ \mbox{and each column}\ k_i\ \mbox{of}\ \lambda\ \mbox{satisfies}\ a_i+2\leq k_i\leq a_i+3\}$

Briefly we have to compute diagrams of type $H_{\underline{a}}^+(V)$ such that when we add two boxes (not on the same column) they have one of the two previous forms.\\
First of all we analyze contributes of the first type. We want to prove that their number is $n2^n$. The case $n=1$ is obvious so let us suppose $n>1$.\\
Let us start by studying the case $a_1<2n-1=2(n-1)+1$. In this case we have that the contribute is given by the odd orthogonal case for $2(n-1)+1$ and so
(see \cite{Dolce}) it is $(n-1)2^n$.\\
Let us remark that this time we could do another operation with respect to that case: when $a_1=2n-2$ we could add a box to the first row, but in this case we need $a_2=2n-3$ and then we must add the second box to the second row

$$
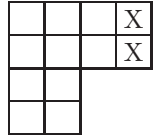
\begin{Young}
      &&&X\cr
     &&&X \cr
     & \cr
     &\cr
\end{Young}$$

but with this operation we have two boxes on the same column. So we have no contributes of this type.

The other possibility is $a_1=2n-1$. There are only two way to add boxes.
In each case we have to add the first box to the first column and when
 $a_2=2n-2$ we have to add the second box to the second column
 
 $$
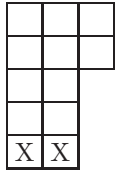
\begin{Young}
      &&\cr
     && \cr
     & \cr
     &\cr
     X&X\cr
\end{Young}$$

so in this case we have to compute subdiagrams with even columns of type 
$2n-2=2(n-1)>b_1>...$ and we know that they are $2^{n-1}$.

When $a_2=2n-3$ we have to add the second box to the second row
$$
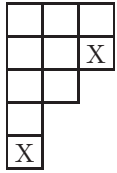
\begin{Young}
      &&\cr
     &&X \cr
     & \cr
     \cr
     X\cr
\end{Young}$$

and we have to compute subdiagrams of type $2n-3>b_1>...$ which are $2^{n-1}$. By adding contributes we have $n2^n$. 

Let us consider now contributes given by diagrams of the second type.

We preliminarily distinguish two cases. One is starting from a row of maximal length ($2n$), we are going to discuss this case later.\\
The second case is when the first row is not of maximal length, then the only cases that we can consider are: $a_1=2n-2$ and $a_1=2n-3$.\\
If $a_1=2n-2$ then by adding a box to the first row and then by deleting it we obtain a diagram with the first row of length $2n-1$ and we have to add the second box to the first column. So we have to compute subdiagrams of type $2n-2=2(n-1)>b_1>...$ with even columns that we know to be $2^{n-1}$.\\
If $a_1=2n-3$ by adding two boxes to the first row and then deleting it we have to compute subdiagrams of type $2n-3>b_1>...$ with even columns which are $2^{n-1}$.  \\
Let us come back to $a_1=2n-1$. We proceed by induction. We state that they are $(n-1)2^{n}$.\\
We are considering diagrams $H_{\underline{a}}(V)$ with $\underline{a}$ of type $2n-1>b_1>...$ in which we added on the left a column of length $2n$.\\
We start by considering $2n-3=2(n-2)+1>b_1>...$, in this case the only a way to add two boxes is to add them in the interior subdiagram and we can compute by the orthogonal case, so we have $(n-2)2^{n-1}$ contributes.\\
There are other three cases to study. The first one is $b_1=2n-2$, in this case we have this type of diagram:

$$
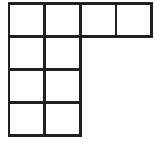
\begin{Young}
      &&&\cr
     & \cr
     & \cr
     &\cr
\end{Young}$$

so if $b_2=2n-3$ than we have to add boxes in the interior part and by induction they are $(n-2)2^{n-1}$.\\
If $b_2=2n-4$ then we have to add the boxes to the second row and the second column, so we have $ 2^{n-2}$. 
$$
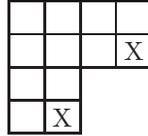
\begin{Young}
      &&&\cr
     &&&X \cr
     & \cr
     &X\cr
\end{Young}$$

Finally if $b_2=2n-5$ we have to add two boxes to the second row and we have $2^{n-2}$.\\ 
The second and the last case is $b_1=2n-3$. Now we have to add the first box to the second column
$$
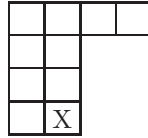
\begin{Young}
      &&&\cr
     & \cr
     & \cr
     &X\cr
\end{Young}$$

and then if $b_2=2n-4$ we have to add the other box to the third column and we have $2^{n-2}$ contributes 
$$
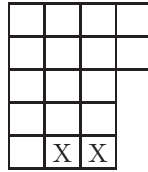
\begin{Young}
      &&&\cr
     &&& \cr
     && \cr
     &&\cr
     &X&X\cr
\end{Young}$$

else $b_{2}=2n-5$ then we have to add the second box to the second row
$$
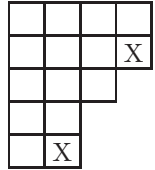
\begin{Young}
      &&&\cr
     &&&X \cr
     && \cr
     &\cr
     &X\cr
\end{Young}$$

and we have $ 2^{n-2}$. By adding the previous contributes we have proved the induction.

\end{proof}
\end{proposition}

\subsection{New types of invariants in $(\bigwedge (M_{2n}^+)^*\otimes M_{2n}^-)^G$}

It is a classical result that the determinant of a skew-symmetric matrix $X$ can be expressed as a square of a polynomial $Pf(X)$, the ``pfaffian'', in the entries of this matrix. Furthermore we can extend this function to a generic matrix $Y$ (with an abuse of notation still called $Pf$) by $Pf(Y):=Pf(\frac{Y-Y^t}{2})$. So we can consider the polarized pfaffian
$$Q(Y_1,...,Y_n):=Pf_L\left(\frac{Y_1-Y_1^t}{2},...,\frac{Y_n-Y_n^t}{2}\right),$$ and the element
\begin{equation}\label{LAQ}Q([X_1,X_2],...,[X_{2n-3},X_{2n-2}],X)\in \left(\bigotimes^{2n-2}M_{2n}^+\otimes M_{2n}^-\right)^*.\end{equation}
Let us remark that this element is $SO(2n)-$invariant and if $X_i=\lambda Id$ we have that it equals to zero, so, by linearity, we can reduce our computations on the spaces of traceless matrices.

Thanks to the natural isomorphism with the dual space, \eqref{LAQ} corresponds to the element 
$$\sum_{i<j} Q([X_1,X_2],...,[X_{2n-3},X_{2n-2}],e_i\wedge e_j)e_i\wedge e_j \in \left(\bigotimes^{2n-2}(M_{2n}^+)^*\otimes M_{2n}^-\right)^{SO(2n)},$$
where $e_1,...,e_{2n}$ is an orthonormal basis of $\mathbb{C}^{2n}$.\\
Let us consider $\Omega(X_1,...,X_{2n-2})\in \left(\bigwedge^{2n-2}(M_{2n}^+)^*\otimes M_{2n}^-\right)^{SO(2n)}$, defined by
$$\Omega(X_1,...,X_{2n-2}):=\sum_{i<j,\sigma\in S_{2n-2}} Q([X_{\sigma(1)},X_{\sigma(2)}],...,[X_{\sigma(2n-3)},X_{\sigma(2n-2)}],e_i\wedge e_j)e_i\wedge e_j$$
We can think $\Omega$ as an element of 
$$\mbox{hom}_{\mathfrak{so}(2n)}\left(\bigwedge^{2n-2}\frac{\mathfrak{sl}(2n)}{\mathfrak{so}(2n)},\mathfrak{so}(2n)\right).$$
Define\footnote{This is the relative differential in the space $\mbox{hom}_{\mathfrak{so}(2n)}\left(\bigwedge^*\frac{\mathfrak{sl}(2n)}{\mathfrak{so}(2n)},\mathfrak{so}(2n)\right)$ which produces an other invariant. For more detail you can see \cite{Kumar} Chapter 3.} $d\Omega(X_1,...,X_{2n-1})\in \mbox{hom}_{\mathfrak{so}(2n)}\left(\bigwedge^{2n-1}\frac{\mathfrak{sl}(2n)}{\mathfrak{so}(2n)},\mathfrak{so}(2n)\right)$  by
\begin{equation}\label{reldif}
\sum_{i<j}(-1)^{i+j+1} \Omega([X_i,X_j],X_1,...,\hat{X}_i,...,\hat{X}_j,...,X_{2n-1})+\sum_h (-1)^hX_h\cdot \Omega(X_1,...,\hat{X}_h,...,X_{2n-1}),
\end{equation}
where $X\in \mathfrak{sl}(2n)$ acts on $\mathfrak{so}(2n)$ by
$$X\cdot Y=XY+YX^t.$$
We remark that the first addend of \eqref{reldif} vanishes because the bracket between a symmetric matrix and a skew symmetric matrix is a symmetric matrix. So
$$d\Omega(X_1,...,X_{2n-1})=$$
$$\sum_h (-1)^h X_h\cdot\left(\sum_{i<j,\sigma_h\in S_{2n-2}^h} \epsilon(\sigma_h)Q\left([X_{\sigma_h(1)},X_{\sigma_h(2)}]...,\hat{X}_h,...,[X_{\sigma_h(2n-2)},X_{\sigma_h(2n-1)}],e_i\wedge e_j\right)e_i\wedge e_j\right),$$
where $S_{2n-2}^h$ is the symmetric group of $\{1,...,\hat{h},...,2n-1\}$.

Our  goal is to show that $\mbox{Tr}(\Omega\wedge d\Omega)$ is, up to a rest $R$, a non zero multiple of $ \mbox{Tr}(X^{4n-3})$. To achieve this, we will use the elementary matrices $u\otimes v$, $u,v\in V$ (let us recall that elementary matrices $u\otimes u$, with $u$ isotropic vector, generate the space of traceless symmetric matrices).  

We start by recalling two simple, well known facts. First of all given vectors $u_i,v_i\in V$, $i=1,\ldots ,2n$,
\begin{equation}\label{suivett}Q(u_1\otimes u_2,...,u_{2n-1}\otimes u_{2n})=[u_1,u_2,...,u_{2n-1},u_{2n}]\end{equation} and since  clearly
$$[u_1,u_2,...,u_{2n-1},u_{2n}][v_1,v_2,...,v_{2n-1},v_{2n}]=
det[(u_i,v_j)],$$
 we deduce
\begin{equation}\label{ILPRIMO}Q(u_1\otimes u_2,...,u_{2n-1}\otimes u_{2n})Q(v_1\otimes v_2,...,v_{2n-1}\otimes v_{2n})=det[(u_i,v_j)]=\sum_{\tau\in S_{2n}}\epsilon(\tau)\Pi_{i=1}^{2n}(u_i,v_{\tau(i)}).\end{equation}

Secondly, for $u_1,\ldots u_{4n-3}\in V$,
\begin{equation}\label{ILSECONDO}\mbox{Tr}(u_{1}\otimes u_{1}\cdot...\cdot u_{4n-3}\otimes u_{4n-3})=(u_{1},u_{2})(u_{2},u_{3})\cdots(u_{4n-3},u_{1}).\end{equation}

At this point let us write $\Omega\wedge d\Omega$ explicitly. We get
$$\Omega\wedge d\Omega(X_1,...,X_{4n-3})=$$
$$\sum_{i<j,\ l<m,\ h,\ \sigma_h,\sigma\in S_{2n-2},\ \tau\in S_{4n-3}}(-1)^{\tau(h)}\epsilon(\tau)\epsilon(\sigma)\epsilon(\sigma_h) Q([X_{\sigma(\tau(1))},X_{\sigma(\tau(2))}],...,[X_{\sigma(\tau(2n-3))},X_{\sigma(\tau(2n-2))}],e_i\wedge e_j)$$
$$Q\left([X_{\sigma_h(\tau(2n-1))},X_{\sigma_h(\tau(2n))}]...,\hat{X}_{\tau(h)},...,[X_{\sigma_h(\tau(4n-4))},X_{\sigma_h(\tau(4n-3))}],e_l\wedge e_m\right)e_i\wedge e_j(X_{\tau(h)} e_l\wedge e_m+e_l\wedge e_m X_{\tau(h)}).$$
Up to a non zero scalar, this equals 
$$\sum_{i<j,\ l<m,\ h,\ \tau\in S_{4n-3}}(-1)^{\tau(h)}\epsilon(\tau) Q\left([X_{\tau(1)},X_{\tau(2)}],...,[X_{\tau(2n-3)},X_{\tau(2n-2)}],e_i\wedge e_j\right)$$
$$Q\left([X_{\tau(2n-1)},X_{\tau(2n)}]...,\hat{X}_{\tau(h)},...,[X_{\tau(4n-4)},X_{\tau(4n-3)}],e_l\wedge e_m\right)e_i\wedge e_j(X_{\tau(h)} e_l\wedge e_m+e_l\wedge e_m X_{\tau(h)}).$$
By the  linearity of $Q$ we then deduce that this  equals
$$\sum_{i<j,\ l<m,\ h,\ \tau\in S_{4n-3}}(-1)^{\tau(h)}\epsilon(\tau) Q\left([X_{\tau(1)},X_{\tau(2)}],...,[X_{\tau(2n-3)},X_{\tau(2n-2)}],e_i\wedge e_j\right)$$
$$Q\left([X_{\tau(2n-1)},X_{\tau(2n)}]...,\hat{X}_{\tau(h)},...,[X_{\tau(4n-4)},X_{\tau(4n-3)}],(X_{\tau(h)} e_l\wedge e_m+e_l\wedge e_m X_{\tau(h)})\right)e_i\wedge e_j e_l\wedge e_m.$$

So, taking the trace, we have up to a non zero scalar,
$$\mbox{Tr}(\Omega\wedge d\Omega)=$$
$$\sum_{i<j,\ h,\ \tau\in S_{4n-3}}(-1)^{\tau(h)}\epsilon(\tau) Q\left([X_{\tau(1)},X_{\tau(2)}],...,[X_{\tau(2n-3)},X_{\tau(2n-2)}],e_i\wedge e_j\right)$$
$$Q\left([X_{\tau(2n-1)},X_{\tau(2n)}]...,\hat{X}_{\tau(h)},...,[X_{\tau(4n-4)},X_{\tau(4n-3)}],(X_{\tau(h)} e_i\wedge e_j+e_i\wedge e_j X_{\tau(h)})\right)$$
which, again up to a non zero scalar, we can rewrite  as
\begin{equation}\label{lollo}
\sum_{i,j,\ \tau\in S_{4n-3}}\epsilon(\tau) Q\left([X_{\tau(1)},X_{\tau(2)}],...,[X_{\tau(2n-3)},X_{\tau(2n-2)}],e_i\wedge e_j\right)
\end{equation}
$$
\ \ \ \ \ \ \ Q\left([X_{\tau(2n-1)},X_{\tau(2n)}]...,[X_{\tau(4n-3)},X_{\tau(4n-2)}],(X_{\tau(4n-3)} e_i\wedge e_j+e_i\wedge e_j X_{\tau(4n-3)})\right).
$$

We will complete our computation considering $X_{i}=u_{i}\otimes u_{i}$, with $u_i$ isotropic vector. First we make two preliminary remarks:
\begin{enumerate}
\item $[u\otimes u,v\otimes v]=(u,v)u\wedge v;$
\item $\begin{aligned}u\otimes u\cdot (e_i\otimes e_j-e_j\otimes e_i)&= u\otimes u(e_i\wedge e_j)+(e_i\wedge e_j)u\otimes u\\
&=(u,e_i)u\wedge e_j-(u,e_j)u\wedge e_i.\end{aligned}$
\end{enumerate}

So we can substitute \eqref{lollo} by
$$
\sum_{i,j,\ \tau\in S_{4n-3}}\epsilon(\tau) (u_{\tau(1)},u_{\tau(2)})(u_{\tau(3)},u_{\tau(4)})\cdots(u_{\tau(4n-5)},u_{\tau(4n-4)})Q(u_{\tau(1)}\wedge u_{\tau(2)},...,u_{\tau(2n-3)}\wedge u_{\tau(2n-2)},e_i\wedge e_j)$$ 

$$[(u_{\tau(4n-3)},e_i)Q(u_{\tau(2n-1)}\wedge u_{\tau(2n)},...,u_{\tau(4n-5)}\wedge u_{\tau(4n-4)},u_{\tau(4n-3)}\wedge e_j)-\ \ \ \ \ \ \ \ $$ 
$$\ \ \ \ \ \ \ \ \ \ \ \ \ \ \ \ \ \ \ (u_{\tau(4n-3)},e_j)Q(u_{\tau(2n-1)}\wedge u_{\tau(2n)},...,u_{\tau(4n-5)}\wedge u_{\tau(4n-4)},u_{\tau(4n-3)}\wedge e_i)]$$

that, up to a non zero  scalar, equals
$$
\sum_{i,j,\ \tau\in S_{4n-3}}\epsilon(\tau) (u_{\tau(1)},u_{\tau(2)})(u_{\tau(3)},u_{\tau(4)})\cdots(u_{\tau(4n-5)},u_{\tau(4n-4)})(u_{\tau(4n-3)},e_i)
$$

\begin{equation}\label{latre}
Q(u_{\tau(1)}\wedge u_{\tau(2)},...,u_{\tau(2n-3)}\wedge u_{\tau(2n-2)},e_i\wedge e_j)Q(u_{\tau(2n-1)}\wedge u_{\tau(2n)},...,u_{\tau(4n-5)}\wedge u_{\tau(4n-4)},u_{\tau(4n-3)}\wedge e_j).
\end{equation}
But thanks to \eqref{ILPRIMO} we have that a single summand of \eqref{latre}   equals

$$\epsilon(\tau) (u_{\tau(1)},u_{\tau(2)})\cdots(u_{\tau(4n-3)},e_i)\left|\begin{matrix} 
(u_{\tau(1)},u_{\tau(2n-1)}) & \cdots & (u_{\tau(1)},u_{\tau(4n-3)})& (u_{\tau(1)},e_j) \\
 \vdots&  &\vdots & \vdots\\
(u_{\tau(2n-2)},u_{\tau(2n-1)}) & \cdots &(u_{\tau(2n-2)},u_{\tau(4n-3)} & (u_{\tau(2n-2)},e_j) \\
(e_i,u_{\tau(2n-1)}) &\cdots  &(e_i,u_{\tau(4n-3)}) &(e_i,e_j) \\
(e_j,u_{\tau(2n-1)}) &\cdots  &(e_j,u_{\tau(4n-3)}) &(e_j,e_j)
\end{matrix}\right|
$$

by adding on $i$ in \eqref{latre} we have 
\begin{equation}\label{sviluppata}
\sum_{j,\tau}\epsilon(\tau) (u_{\tau(1)},u_{\tau(2)})\cdots(u_{\tau(4n-5)},u_{\tau(4n-4)})\left|\begin{matrix} 
(u_{\tau(1)},u_{\tau(2n-1)}) & \cdots & (u_{\tau(1)},u_{\tau(4n-3)})& (u_{\tau(1)},e_j) \\
 \vdots&  &\vdots & \vdots\\
(u_{\tau(2n-2)},u_{\tau(2n-1)}) & \cdots &(u_{\tau(2n-2)},u_{\tau(4n-3)} & (u_{\tau(2n-2)},e_j) \\
(u_{\tau(4n-3)},u_{\tau(2n-1)}) &\cdots  &0&(u_{\tau(4n-3)},e_j) \\
(e_j,u_{\tau(2n-1)}) &\cdots  &(e_j,u_{\tau(4n-3)}) &(e_j,e_j)
\end{matrix}\right|.\end{equation}
By the Laplace expansion for the determinant with respect to the last row, we have that the determinant in the previous formula equals to
$$\sum_{k=0}^{2n-2}(-1)^k(e_j,u_{2n-2+k})\left|\begin{matrix} 
(u_{\tau(1)},u_{\tau(2n-1)}) & \cdots &\widehat{(u_{\tau(1)},u_{\tau(2n-2+k)})}&\cdots & (u_{\tau(1)},u_{\tau(4n-3)})& (u_{\tau(1)},e_j) \\
 \vdots& &\vdots & &\vdots & \vdots\\
(u_{\tau(2n-2)},u_{\tau(2n-1)}) & \cdots &\widehat{(u_{\tau(2n-2)},u_{\tau(2n-2+k)})}&\cdots&(u_{\tau(2n-2)},u_{\tau(4n-3)}) & (u_{\tau(2n-2)},e_j) \\
(u_{\tau(4n-3)},u_{\tau(2n-1)}) &\cdots&\widehat{(u_{\tau(4n-3)},u_{\tau(2n-2+k)})}&\cdots  & 0&(u_{\tau(4n-3)},e_j) 
\end{matrix}\right|
$$
$$\ \ \ \ \ \ \ \ \ \ \ \ \ \ (-1)^{2n-1}\left|\begin{matrix} 
(u_{\tau(1)},u_{\tau(2n-1)}) & \cdots & (u_{\tau(1)},u_{\tau(4n-3)}) \\
 \vdots&  &\vdots \\
(u_{\tau(2n-2)},u_{\tau(2n-1)}) & \cdots &(u_{\tau(2n-2)},u_{\tau(4n-3)})  \\
(u_{\tau(4n-3)},u_{\tau(2n-1)}) &\cdots  & 0
\end{matrix}\right|
$$

where in the first matrix the column $(u_{\tau(r)},u_{\tau(2n-2+k)})$ is missing. By adding over $j$ and permuting the columns, all the resulting determinants equal to
$$
-\left|\begin{matrix} 
(u_{\tau(1)},u_{\tau(2n-1)}) & \cdots & (u_{\tau(1)},u_{\tau(4n-3)}) \\
 \vdots&  &\vdots \\
(u_{\tau(2n-2)},u_{\tau(2n-1)}) & \cdots &(u_{\tau(2n-2)},u_{\tau(4n-3)})  \\
(u_{\tau(4n-3)},u_{\tau(2n-1)}) &\cdots   & 0

\end{matrix}\right|,
$$
expanding with respect to the last row we have
$$
\sum_{l=0}^{2n-3}(-1)^{l+1}(u_{\tau(4n-3)},u_{\tau(2n-1+l)})\left|\begin{matrix} 
(u_{\tau(1)},u_{\tau(2n-1)}) & \cdots &\widehat{(u_{\tau(1)},u_{\tau(2n-2+l)})}&\cdots & (u_{\tau(1)},u_{\tau(4n-3)}) \\
 \vdots& & & &\vdots\\
(u_{\tau(2n-2)},u_{\tau(2n-1)}) & \cdots  &\widehat{(u_{\tau(2n-2)},u_{\tau(2n-2+l)})}&\cdots &(u_{\tau(2n-2)},u_{\tau(4n-3)})  \\
\end{matrix}\right|
$$

So \eqref{sviluppata} becomes

\begin{equation}\label{gfrt}\sum_{ l,\tau}(-1)^{l+1}\epsilon(\tau)(u_{\tau(1)},u_{\tau(2)})\cdots(u_{\tau(4n-5)},u_{\tau(4n-4)})(u_{\tau(4n-3)},u_{\tau(2n-1+l)})$$ $$ \sum_{\sigma\in S_{2n-3}^{2n-1+l}}\epsilon(\sigma) (u_{\tau(1)},u_{\sigma(\tau(2n-1))})\cdots (u_{\tau(2n-2)},u_{\sigma(\tau(4n-3))}), \end{equation}
When we express the above element as an expression in  the traces of $X^i$, the element of the form
$(u_{4n-3}, u_1)(u_{1},u_2)\cdots (u_{4n-4},u_{4n-3})$ contributes only to $Tr(X^{4n-3})$. We are going presently to show that their coefficient in the above expression is non zero.

 Notice now that   a summand in \eqref{gfrt} equals $(u_{4n-3}, u_1)(u_{1},u_2)\cdots (u_{4n-4},u_{4n-3})$ only if is of the form 
 $$[(u_{h},u_{h+1})(u_{h+2},u_{h+3})\cdots (u_{h-1},u_h)][(u_{h+1},u_{h+2})(u_{h+3},u_{h+4})\cdots (u_{h-2},u_{h-1})$$
 where we order the elements $1,\ldots ,4n-3$ cyclically modulo  $4n-3$.
 For a given $h$ there are exactly two such monomials, the first is the one write above, while the second is given applying the involution mapping $h+i$ to $h-i$ which since $4n-3$ is congruent to $1$ modulo 4 has positive sign. 
 So for each given $h$ this monomial appear either with coefficient $\pm 2$. Since the number of $h$ is odd our claim follows and we have shown

\begin{teorem}
With the previous notation, we have
\begin{equation}\label{pairingpfaffiano}
Tr(\Omega\wedge d\Omega)=q\cdot Tr(X^{4n-3})+R,\end{equation}
where $q\in\mathbb{Q}-\{0\}$ and $R$ is a homogeneous polynomial of degree $4n-3$ in the elements $Tr(X^i)$ with $i<4n-3$.
\end{teorem}

\subsection{Explicit description}
We are ready to develop the main part of the paper: the explicit description of the the space $\left(\bigwedge (M_{2n}^+)^*\otimes M_{2n}^-\right)^{SO(2n)}$.

We set $T_i:=Tr(St_{4i+1}(X_1,...,X_{4i+1}))$ and 
$$Q:=\sum_{\sigma\in S_{2n}}Q([X_{\sigma(1)},X_{\sigma(2)}],...,[X_{\sigma(2n-1)},X_{\sigma(2n)}]).$$ Let us start with a technical lemma:
\begin{lemma}
\begin{equation}\label{QQ}
Q\wedge Q=0.\end{equation}
\end{lemma}
\begin{proof}
We can prove our lemma by considering $X_i=u_i\otimes u_i$ as in the previous section.\\
We have that $Q\wedge Q$ is the antisymmetrization of the element: 
\begin{equation}\label{ELEMENTO}
(u_1,u_2)\cdots (u_{4n-1},u_{4n})det(u_i,u_{2n+j}),\end{equation}
but we can observe that, up to a sign, \eqref{ELEMENTO} is a sum of elements of type: ''products of cycles of even length'', where a cycle is an element of type $(u_{i_1},u_{i_2})(u_{i_2},u_{i_3})\cdots (u_{i_k},u_{i_1})$ and the length is the number of scalar products. So by the formula \eqref{ILSECONDO} each summand corresponds to the product of traces of an even monomial and when we antisymmetrize we obtain zero.
\end{proof}

We can deduce a well known result (\cite{Bo2},\cite{Tak}).
\begin{proposition}
The space $\left(\bigwedge (M_{2n}^+)^*\right)^G$ is the symmetric product of the exterior algebra in the elements $T_i$ (where $h=0,...,n-1$) and the one dimensional algebra generated by the element $Q$.
\end{proposition}
\begin{proof}
By the Lemma \ref{evendim} we have that the dimension of this space is $2^{n+1}$. So the only thing that we have to prove is the linear independence.\\
We know that the multilinear invariant functions are generated by traces of monomials and the polarized pfaffian of $2n$ monomials of matrices. Further we recall, by \cite{Dolce}, that $Tr(St_i)=0$ if $i\neq 4h+1$. If we consider the element:
\begin{equation}\label{ilprodotto}
T_0\wedge T_1\wedge \cdots \wedge T_{n-1} \wedge Q
\end{equation}
we see that this element has maximum degree and it is the only way to form an element of such degree (we recall that by the Amitsur-Leviski's Theorem $Tr(St_h)=0$, $h\ge 4n$).\\
Further we know that in degree maximum there is an invariant (thanks to the description by diagrams), so \eqref{ilprodotto} is non zero and we have the thesis since $T_i\wedge T_i=Q\wedge Q=0.$
\end{proof}

We denote by $A_n\subset \left(\bigwedge (M_{2n}^+)^*\right)^G$ the subalgebra generated by $T_0,...,T_{n-2},Q$. Then we have:
\begin{teorem}\label{CASOPARI}
The space $\left(\bigwedge (M_{2n}^+)^*\otimes M_{2n}^-\right)^G$ is a free module on the subalgebra $A_{n}\subset \left(\bigwedge (M_{2n}^+)^*\right)^G$ with basis $X^2,X^3,...,X^{4h+2},X^{4h+3},...,X^{4n-6},X^{4n-5},\Omega,d\Omega$. 
\end{teorem}
\begin{proof}
By the Proposition \ref{dimpari}, the only thing that we have to prove is the linearly independence of the generators on the algebra $A_n$. So we suppose that
\begin{equation}\label{relation}
0=\sum_{i=0}^{n-2}\left(P_i\wedge X^{4i+2}+Q_{i+1}\wedge X^{4i+3}\right)+A\wedge\Omega+B\wedge d\Omega,\end{equation}
where $P_i,Q_{i+1},A,B\in A_n$. We can see \eqref{relation} as a relation in the space $\left(\bigwedge (M_{2n}^+)^*\otimes M_{2n}\right)^G$, so we can multiply by $X^i$ for all $i\ge0$. Further we can suppose this element is homogeneous and let $j$ be the minimum among all the degrees of the generators which appear with nonzero  coefficient. If this $j$ corresponds to an element of type $X^{j}$ then by multiplying by $X^{4n-3-j}$ and taking the trace we have:
$$0=C\wedge T_{n-1}+A\wedge Tr(\Omega \wedge X^{4n-3-j})+B\wedge Tr(d\Omega\wedge X^{4n-3-j}),$$
but 
$$Tr(\Omega \wedge X^{4n-3-j})$$
and
$$Tr(d\Omega\wedge X^{4n-3-j}) $$
are semiinvariants for the orthogonal group so they can't be multiples of $T_{n-1}$, then thanks to the linear independence showed in the previous proposition we have that the only possibility is $C=0$.\\
If $j$ corresponds to $\Omega$ (or to $d\Omega$) then multiplying by $d\Omega$ (resp. by $\Omega$) and taking the trace thanks to the identity \eqref{pairingpfaffiano} we have the same situation as before\footnote{Let us remark that the degree of $\Omega\wedge \Omega$ is $4n-4$, so its trace can't involve $T_{n-1}$. On the other hand, even if $d\Omega\wedge d\Omega$ has degree $4n-2$, its trace can't involve $T_{n-1}$, otherwise it would have a summand of the form $T_{n-1}\wedge T_0$ and this is impossible because we have seen that we can valuate $d\Omega$ on traceless matrices ($T_0=0$).} and so the thesis.
\end{proof}

\subsection{Identities for missing generators}
At this point we are motivated to find out other identities for those elements as 
$X^{4n-2},X^{4n-1},T_{n-1}\wedge X^i$, which don't appear among the generators. Let us recall that by classical theory (\cite{WeRa}) $X^{4n-2}$ and $X^{4n-1}$ are not zero.

\begin{proposition}
\begin{equation}\label{pera}
X^{4n-2}=q_1\cdot Q\wedge \Omega,
\end{equation}
\begin{equation}\label{mela}
X^{4n-1}=q_2\cdot Q\wedge d\Omega,
\end{equation}
where $q_1,q_2$ are non zero scalar.
\end{proposition}

\begin{proof}
Let us prove \eqref{pera}. By the Theorem \ref{CASOPARI}, we have $X^{4n-2}=q_1\cdot Q\wedge \Omega+R_1$, with $R_1$ of type $\sum P_i\wedge X^i$, $i<4n-2$ and $P_i$ of type ''trace''. Now we want to prove $R_1=0$. 
In fact, if we have:
$$ X^{4n-2}=q_1 Q\wedge \Omega+ \sum P_i\wedge X^i,$$
multiplying by $X^{4n-3-i}$ and taking the trace we obtain $P_i=0$. So $q_1$ must be non zero.\\ 
The prove of \eqref{mela} is similar.
\end{proof}
We still need another identity for the element $T_{n-1}\wedge X^2$. First of all let us prove two preliminary results:
\begin{lemma}
\begin{equation}\label{duale1}
Tr(\Omega\wedge X^2)=-Q,
\end{equation}
\begin{equation}\label{duale2}
Tr(d\Omega \wedge X^2)=0. 
\end{equation}
\end{lemma}
\begin{proof}
We will prove that 
\begin{equation}\label{duale3}
Tr\left(\sum_{i,j}Q([X_1,X_2],...,[X_{2n-3},X_{2n-2}],e_i\wedge e_j)e_i\wedge e_j\cdot X_{2n-1}X_{2n-2}\right)=2Q,\end{equation} and \eqref{duale1} will follow.\\
In fact if we consider $X_{i}=u_{i}\otimes u_{i}$ we have that \eqref{duale3} is
$$2\sum_{i,j} (u_1,u_2)\cdots (u_{2n-1},u_{2n})[u_1,...,u_{2n-2},e_i,e_j](e_j,u_{2n-1})(e_i,u_{2n}),$$
so we can include $(e_j,u_{2n-1})(e_i,u_{2n})$ in the determinant and summing over $i,j$ we have the thesis.\\
To prove the second formula we can remark that, since the degree of $d\Omega\wedge X^2$ is $2n+1$ and it is a semiinvariant, if $Tr(d\Omega\wedge X^2)\neq 0$ it must be $kQ\wedge T_0$, with $k$ a non zero scalar. We know that, up to a scalar, we can write $d\Omega$ as the antisymmetrization of the element:
$$\sum_{i,j}Q([X_1,X_2],...,[X_{2n-3},X_{2n-2}],X_{2n-1}\cdot e_i\wedge e_j)e_i\wedge e_j,$$
so if we consider 
$$\sum_{i,j}Q\left([X_1,X_2,...,[X_{2n-3},X_{2n-2}],X_{2n-1}e_i\wedge e_j+e_i\wedge e_j X_{2n-1}\right)e_i\wedge e_j X_{2n}X_{2n+1},$$
and we consider $X_i=u_i\otimes u_i$ and we take the trace, we have (up to a scalar)
$$\sum_{i,j}(u_1,u_2)\cdots (u_{2n-3},u_{2n-2})(u_{2n},u_{2n+1})[u_1,...,u_{2n-1},e_i](u_{2n-1},e_j)(u_{2n}, e_i)(u_{2n+1},e_j),$$
summing on $i,j$ we can rewrite it as
$$\sum_{i,j}(u_1,u_2)\cdots (u_{2n-3},u_{2n-2})(u_{2n},u_{2n+1})[u_1,...,u_{2n-1},u_{2n}](u_{2n-1},u_{2n+1}).$$
We see that there is not an element of type $(u,u)$ and so when we antisymmetrize we can't obtain $T_0$.
\end{proof}
So we can deduce:
\begin{proposition}
$$T_{n-1}\wedge X^2=-\sum T_i\wedge X^{4(n-i)-2}+kQ\wedge d\Omega,$$
with $k$ a non zero scalar.
\end{proposition}
\begin{proof}
We can deduce by the Theorem \ref{CASOPARI} that there is an identity of type:
$$T_{n-1}\wedge X^2=\sum P_i\wedge X^i +A\wedge Q\wedge \Omega +kQ\wedge d\Omega,$$
where $A,P_i\in A_n$ and $k$ a scalar. First of all we look at $k$ and $A$. If we multiply by $\Omega$ and we take the trace we have by \eqref{duale1}, \eqref{QQ} and \eqref{pairingpfaffiano}
$$T_{n-1}\wedge Q=kqQ\wedge T_{n-1} +R,$$
with $R\in A_n$, so $k\neq 0$. On the other hand if we multiply by $\Omega$ and we take the trace we have:
$$0=A\wedge Q\wedge T_{n-1}+S,$$
with $S\in A_n$, so $A=0$. Now we pass to determinate $P_i$. We first remark that $Tr(St_{4n-3-i+2})\neq 0 \iff i=2\ \mbox{mod}\, 4$. So if we multiply by $X^{4n-3-i}$, with $i=4k+2$ and we take the trace, we have 
$$T_{n-1}\wedge T_{n-k-1}=\sum P_{4k+2}\wedge T_{n-1},$$
than our claim follows.
\end{proof}

\section{Cases of type $B^-$}

In this section we will consider the cases of type $B^-$. We will see, reporting a computation made by a computer (using the software: Lie), they are not free module on the corresponding subalgebra of the invariants. To do the computation we need traceless matrices. So our computation is made for this kind of matrices. We denote traceless matrices by $N_{2n}$.

Let us start considering the case $\left(\bigwedge (N_{2n}^+)^*\otimes N_{2n}^+\right)^G$. If $n=3$, we have the following table for the multiplicity of $N_{2n}^+$ in $\bigwedge (N_{2n}^+)^*$ (let us recall that, by the Poincar\'e duality, we need to know the multiplicity only until the degree $10=dim(N_{2n}^+)/2$):  
\begin{center}
\begin{tabular}{r|c|c|c|c|c|c|c|c|c|c|c|c|c|c|c|}
degree      &0&1&2&3&4&5&6&7&8&9&10\\ \hline
multiplicity&0&1&0&0&1&2&2&1&1&2&4\\ \hline
\end{tabular}
\end{center}

So the Poincar\'e polynomial is 
$$t^1+2t^5+2t^6+t^7+2t^9+4t^{10}+2t^{11}+t^{12}+t^{13}+2t^{14}+2t^{15}+t^{16}+t^{19},$$
which is not divisible by $(1+t^5)(1+t^6)$ and our claim follows.

The second case is $\left(\bigwedge (N_{2n}^-)^*\otimes N_{2n}^+\right)^G$. We have for $n=4$ that the multiplicity of $N_{2n}^+$ in $\bigwedge (N_{2n}^-)^*$ is given by the following table:
\begin{center}
\begin{tabular}{r|c|c|c|c|c|c|c|c|c|c|c|c|c|c|c|}
degree&0&1&2&3&4&5&6&7&8&9&10&11&12&13&14\\ \hline
multiplicity&0&0&0&1&1&0&2&3&1&1&4&4&1&3&6\\ \hline
\end{tabular}
\end{center}

Let us remark that we need to compute only until the degree $14$, because by the Poincar\'e duality we can deduce the remaining numbers.

We have that the Poincar\'e polynomial is:
 $$t^3+t^4+2t^6+3t^7+t^8+t^9+4t^{10}+4t^{11}+t^{12}+3t^{13}+6t^{14}+3t^{15}+t^{16}+4t^{17}+4t^{18}+t^{19}+t^{20}+3t^{21}+2t^{22}+t^{24}+t^{25}$$
which is not divisible by $(1+t^3)(1+t^4)(1+t^7)$. By classic results on the structure of invariants we have our claim.
\section*{Acknowledgements}
I am grateful to professors C.De Concini and C.Procesi for useful suggestions and advices.


\begin{thebibliography}{9} 
\addcontentsline{toc}{chapter}{Bibliography} 
\bibitem{AL}  S.Amitsur, J.Levitzki, {\em Minimal identities for algebras}, Proc. Amer. Math. Soc. 1, (1950),449-463.
\bibitem{Bo}A.Borel, (1967) [1954], Halpern, Edward, ed., {\em Topics in the homology theory of fibre bundles}, Lecture notes in mathematics 36, Berlin, New York: Springer-Verlag
\bibitem{Bo2}A.Borel, {\em Sur la cohomologie des especes fibres principaux et des espaces homogenes de groupes de Lie compacts}, Ann. of Math. 57 (1953), 115-207
\bibitem{BPS} M. Bresar, C.Procesi, S.Spenko, {\em Quasi-identities on matrices and the Cayley-Hamilton polynomial}, Advances in Math. 280 (2015), 439-471,  arXiv:1212.4597
\bibitem{DMP}C. De Concini, P. Moseneder Frajria, P. Papi, C. Procesi, {\em On special covariants in the exterior algebra of a simple Lie algebra}, Atti Accad. Naz. Lincei Cl. Sci. Fis. Mat. Natur. 25 (2014), 331-344. doi: 10.4171/RLM/682,  arXiv:1404.4222
\bibitem{PDP} C.De Concini, P.Papi, C.Procesi, {\em The adjoint representation inside the exterior algebra of a simple Lie algebra}, Advances in Math. 280 (2015), 21-46.
\bibitem{DPP}C.De Concini, P.Papi, C.Procesi, {\em Invariants of E8}, unpublished.
\bibitem{Dolce}S.Dolce, {\em On certain modules of covariants in exterior algebras}, Algebras and Representation Theory (2015), arXiv:1404.2855
\bibitem{Koszul} J.L.Koszul, {\em Homologie et cohomologie des algebres de Lie}, Bull. Soc. Math. France 78 (1950), 66-127.
\bibitem{Kumar} S.Kumar, {\em Kac-Moody Groups, their Flag Varieties and Representation Theory}, (2002) Springer
\bibitem{MD} I.G.MacDonald, {\em Symmetric Functions and Hall Polynomials}. Oxford Mathematical Monographs.
\bibitem{Pr1} C.Procesi, {\em On the theorem of Amitsur-Levitzki}, Israel Journal of Mathematics (2013), arXiv:1308.2421
\bibitem{Pr2} C.Procesi, {\em Lie Groups, an approach through Invariants and Representations}, Universitex, Springer.
\bibitem{Pr3} C.Procesi, {\em The invariant theory of $n\times n$ matrices}, Advances in Math. 19 (1976), 306-381.
\bibitem{Reeder}  M. Reeder, {\em Exterior powers of the adjoint representation}, Canad. J. Math. 49 (1997), 133-159.
\bibitem{Rogora}E.Rogora, {\em Invariants of matrices under the action of the special orthogonal group}, Universit\`a di Roma, "La Sapienza" preprint 10/5, Available on line at:
http./www.mat.uniroma1.it/people/rogora/pdf/lavoro.pdf.
\bibitem{Tak}M.Takeuchi, {\em On Pontrjagin classes of compact symmetric spaces}, J.Fac.Sci.Univ.Tokyo Sect. I 9 1962 313-328 (1962). 
\bibitem{WeRa}M.Wenxin, M.L.Racine, {\em Minimal Identities of Symmetric Matrices}, Transactions of the American Mathematical Society, 1990, 320 pag. 171-192.
\end{thebibliography}
\end{document}